\documentclass[10pt]{amsart}
\usepackage{amssymb, amscd, amsmath, amsthm, epsf, epsfig, latexsym}

\usepackage{pb-diagram}

\def\zz{{\bf Z}}
\def\ff{{\bf F}} 
\def\qq{{\bf Q}}
\def\cc{{\bf C}}

\def\co{\colon\thinspace}

\def\co{\colon}
\def\tD{\tilde{\Delta}}

\newcommand\Aut{\operatorname{Aut}}
\newcommand\Hom{\operatorname{Hom}}
\newcommand{\fig}[2] { \includegraphics[scale=#1]{#2} }

\newcommand{\Ind}{\operatorname{Ind}}
\newcommand{\lk}{\ell k}

\newtheorem{theorem}{Theorem}[section]
\newtheorem{lemma}[theorem]{Lemma}
\newtheorem{corollary}[theorem]{Corollary}
\newtheorem{prop}[theorem]{Proposition}
\theoremstyle{definition}
\newtheorem{definition}[theorem]{Definition}
\def\co{\colon\thinspace}
\def\var{[t^{\pm1}]}
\numberwithin{equation}{section}

\begin{document}

\title[Twisted polynomials and metabelian representations]{Metabelian representations, twisted Alexander polynomials, knot slicing, and mutation}

\author{Chris Herald}
\author{Paul Kirk} 
\author{Charles Livingston}
\thanks{This work was supported in part by the National Science Foundation under Grants 0709625, 0604310 and 0406934}
\address{Chris Herald: Department of Mathematics, University of Nevada, Reno, NV  89557}
\address{Paul Kirk, Charles Livingston: Department of Mathematics, Indiana University, Bloomington, IN 47405 \vskip.2in}
\email{herald@unr.edu}
\email{pkirk@indiana.edu}
\email{livingst@indiana.edu}


\begin{abstract} 
Given a knot complement $X$ and its $p$--fold cyclic cover $X_p\to X$,  we identify twisted polynomials associated to   $GL_1(\ff\var)$ representations of $\pi_1(X_p)$ with twisted polynomials associated to related $GL_p(\ff\var)$ representations of $\pi_1(X)$ which factor through  metabelian representations.

This provides  a simpler  and faster algorithm to compute these polynomials, allowing us to prove that
16 (of 18 previously unknown) algebraically slice knots of 12 or fewer crossings are not slice. We also use this improved algorithm to prove that the 24 mutants of the pretzel knot $P(3,7,9,11,15)$, corresponding to permutations of $(7,9,11,15)$, represent distinct concordance classes.
\end{abstract}

\maketitle
 
\section{Introduction}

In 1975 Casson and Gordon~\cite{cg1}   presented the first examples of algebraically slice knots that are not slice.  Since then, many other powerful obstructions to a knot being slice have been developed, both in the topological locally flat category, the focus of this paper, and in the smooth category.   See~\cite{liv3} for a list of references up 2003.  A few more recent articles include~\cite{chl, cot, fri, liv4, os, ras}.

Despite this remarkable progress   since Levine defined the algebraic concordance group 40 years ago, the challenge of proving that a given algebraically slice knot is not slice has largely remained   intractable.  As evidence, among prime knots of 12 or fewer crossings, there are 18 that are algebraically slice but not readily shown to be slice.  Of these, two fall to the results  of Casson-Gordon concerning 2-bridge knots, but the remaining 16 have been inaccessible until now.  The most recent advances in smooth concordance place 12 crossing knots on the edge of what is computable; in the topological  category the problem is much more difficult.

Here we explore obstructions based on twisted Alexander polynomials and develop readily computable invariants that are highly effective in obstructing sliceness.  In particular, of the 18 knots just mentioned, quick computations demonstrate that 16 are not slice.  Unexpectedly, one of the remaining two knots is shown to be smoothly slice, and only one questionable case remains in the table.

 Our initial work~\cite{kl1, kl2} with twisted knot polynomials began to address the challenge of finding computable slicing obstructions, but that work was not sufficient to effectively deal with any of the outstanding cases taken from the table of 12 crossing knots.  
 
\vskip.1in

\noindent{\bf Twisted Polynomials.} 
Given a space $X$ and a homomorphism    $\rho \co \pi_1(X) \to GL_n(\ff\var)$, where $\ff$ is a field, there is defined a {\it twisted Alexander polynomial}, $\Delta_{X,  \rho}(t) \in \ff\var$.  The early development of this invariant as a tool in classical knot theory, in which case $X$ was taken to be a classical knot complement, appeared in such papers as~\cite{jiang, kitano,  lin, wada}.  The theory and application of twisted knot polynomials has been considered by many authors; a few papers include~\cite{cha, friedl,  goda, hillman, kitano2, silver}.

In~\cite{kl1, kl2, kl3} we considered the case in which $X$ is a cyclic cover of a classical knot complement and $\rho$ is a 1--dimensional complex representation.  One of the main results of~\cite{kl1} was that for appropriately defined $\rho$, $\Delta_{X, \rho}(t)$ can be interpreted as the discriminant of a Casson-Gordon invariant of the knot.  (Discriminants of Casson-Gordon invariants were first studied in~\cite{gl2, lit1}.)  In~\cite{kl2} this was applied to analyze knot concordance problems, for instance distinguishing knots from their reverses in concordance  and distinguishing positive mutants of certain pretzel knots;  discriminants were later used in~\cite{kl3} to further analyze the action of mutation on the concordance group.   
\vskip.1in

\noindent{\bf Results.} 
Our main theoretical result, Theorem \ref {bigrep}, identifies the twisted polynomial developed in~\cite{kl1}, based on a 1--dimensional representation of a cyclic cover of a knot, with  a twisted polynomial associated to a higher-dimensional metabelian representation of the knot group itself.  As a practical matter, this vastly simplifies the computation of twisted polynomials; in brief, the added complexity of working with covers results from the fact that  if a knot group has $g$ generators, then the group of its $n$-fold cyclic branched cover has roughly $ng$ generators. 

The second focus of our theoretical investigations is a detailed analysis of the $\ff_q[\zz_p]$--module structure of $H_1(B_p;\zz_q)$ where $B_p$ denotes the $p$-fold branched cover of a knot in $S^3$. This allows us to identify the space of those characters in $\Hom(H_1(B_p),\zz_q)$ which vanish on equivariant metabolizers for the linking form.   Such characters determine which twisted Alexander polynomials to use to obstruct sliceness. 

As mentioned earlier, our main application is to settle the slice status of all but one of the 18 remaining  algebraically slice knots with 12 or fewer crossings that were not known to be topologically slice.  We show 16 of these are not slice.   A  side note is a construction that proves that one of the 18 is slice.  Though we do not pursue it further in this article, the trick we introduce  should be quite useful in the further enumeration of slice knots.

 There is an interesting parallel between our work and that concerning reversibility of knots.  Fox~\cite{fox} asked in 1961 if nonreversible knots existed, and he pointed to $8_{17}$ as the first case of interest.  (Fox used the word {\it invertible} rather than {\it reversible}.)   Trotter~\cite{trot}
soon showed the existence of nonreversible pretzel knots, but it took almost twenty years before several authors~\cite{hart, kaw} could show that $8_{17}$ is not reversible.  What we find especially satisfying is that Hartley's approach, the first that was capable of addressing general knots in the table, depended on metabelian representations, the same tool that is central here.  To complete the circle, as a  second application we give an example of a 5--stranded pretzel knot for which all 24 of its positive mutants are distinct in concordance. This set consists of twelve knots and their reverses.

\vskip.1in

The authors wish to thank Darrell Haile, Michael Larsen, Swatee Naik,  and Jim Davis for helpful discussions.

\section{Twisted Homology and Polynomials }\label{section2}

Let    $X$ be a finite $CW$--complex with universal cover $\tilde{X}$ and set $\pi = \pi_1(X)$. Our convention is that $\pi$ acts on the left on the cellular chain complex $C_*(\tilde{X})$.  Let $M$ be a right $\zz[\pi]$--module.  

The   twisted chain complex     $C_*(X; M)$ is defined to be $ M \otimes_{\zz[\pi]} C_*(\tilde{X})$.  The twisted homology  of $X$ is given as the homology of this complex:  $H_n(X; M) = H_n( M \otimes_{\zz[\pi]} C_*(\tilde{X}))$.  If $M$ has the compatible structure of a left $S$--module for a ring $S$, so that $M$ is a $(S,\zz[\pi])$-bimodule, then $H_n(X; M)$ inherits a left $S$--module structure.

The order of a cyclic module over a principal ideal domain is the generator of the annihilator ideal.  The order of a direct sum of cyclic modules is the product of the orders of the summands.

\begin{definition}\label{deftwist}  Let $\ff$ denote some field.  Suppose that $M$ is a $(\ff\var, \zz[\pi])$--bimodule. 
Define the {\em twisted Alexander polynomial} associated to $X$ and $M$, $\Delta_{X, M}\in\ff\var$, to be the order of  $H_1(X;M)$ as a left $\ff\var$--module.  This is well-defined up to multiples by units in $\ff\var$, that is, elements of the form $at^k$, $ a\in\ff^*, k \in \zz$.
 If the right $\zz[\pi]$--module structure on $M$ is determined by a homomorphism $\alpha \co \pi\to \Aut(M)$, we sometimes write 
$\Delta_{X,\alpha}$ instead of $\Delta_{X,M}$. 
\end{definition}

In our earlier article \cite{kl1} and all other articles on twisted Alexander polynomials,   the homomorphism $\alpha$ was taken to have the form  $\epsilon\otimes \rho$ for some $\epsilon\co \pi\to \zz$.  More precisely, we took   $V$ an $\ff$--vector space with a  right  $\zz[\pi]$--action determined by a homomorphism $\rho \co \pi \to GL(V)$,  and   constructed the $(\ff\var, \zz[\pi])$--bimodule 
$M=\ff\var\otimes_\ff V$ with the right $\zz[\pi]$--action given by $\alpha=\epsilon\otimes \rho$. In other words, $$(f(t) \otimes v) \cdot \gamma = t^{\epsilon(\gamma)}f(t) \otimes v \rho(\gamma), \text{ for $\gamma$ in $\pi$}. $$ The extra flexibility afforded in Definition \ref{deftwist} by allowing  $\alpha$ to be more general than tensor products of the form $\epsilon\otimes \rho$ permits a streamlining of some of our arguments. In particular, the twisted polynomials denoted $\Delta_{X,\epsilon,\rho}$ in \cite{kl1} are denoted here by $\Delta_{X,\epsilon\otimes \rho}$ (or   $\Delta_{X,\ff\var\otimes_F V}$ if the action is understood).

\section{Shapiro's Lemma}

  Let $ X_p\to X$ be a   degree $p$ connected covering space of $X$, and set $\pi_p=\pi_1(X_p)$.
 Presentations of the group $\pi_p$ become complicated very quickly as $p$ increases, making it difficult to carry out explicit computations of twisted polynomials using  covers.  One  goal  of this article is to identify the twisted polynomial associated with $X_p$ with one associated with $X$.  The basic result of homological algebra needed for doing this is Shapiro's Lemma.  A   discussion  can be found in~\cite{brown}.  
 
 Given a subgroup $H\subset G$ and a right $\zz[H]$--module $M$, the right $\zz[G]$--module  $M\otimes_{\zz[H]}\zz[G]$ is denoted $\Ind_H^G(M)$.  
 
  \smallskip
  
 \noindent{\bf Shapiro's Lemma.}  {\em Let $X_p $ be a connected covering space of $X$, $\pi=\pi_1(X) $ and $\pi_p=\pi_1(X_p)$. Given a   right $\zz[\pi_p]$--module $M$, then $H_i(X_p; M) \cong H_i(X; \Ind_{\pi_p}^{\pi}(M))$.   If $S$ is a ring and $M$ is an  $(S,\zz[\pi_p])$--bimodule, then $\Ind_{\pi_p}^{\pi} (M)$ is an $(S, \zz[\pi])$--bimodule and the homology group isomorphism also preserves the  left $S$--module structure.}
 
   \smallskip
 
 \begin{proof} 
 Let $\tilde X$ denote the universal cover of $X$ and $X_p$. The covering transformations provide the cellular chain complex $C_*=C_*(\tilde X)$  with the structure a  right $\zz[\pi]$--complex.  We have the following isomorphisms of left $S$--chain complexes.
 $$(M \otimes_{\zz[\pi_p]} \zz[\pi]) \otimes_{\zz[\pi]}  C_*  \cong  M \otimes_{\zz[\pi_p]} (\zz[\pi]  \otimes_{\zz[\pi]}  C_* ) \cong M \otimes_{\zz[\pi_p]} C_* .$$
These induce isomorphisms on homology. \end{proof}

The module $\Ind_H^G(M)$ can also be described as follows (see~\cite{brown} for details). First,  
$\zz[G]$ is a free left $\zz[H]$--module on $H\backslash G$.  Let $R=\{g_i\}\subset G$ be a complete set of coset representatives; these form a basis for $\zz[G]$ as a left $\zz[H]$--module, and hence as an abelian group, 
\begin{equation}\label{split} \Ind_H^G(M)=M\otimes_{\zz[H]}\zz[G]\cong \bigoplus _{g_i\in R}M\otimes g_i, \end{equation} as an internal direct sum.
  The right $\zz[G]$--action on the direct sum is described in this basis as follows: if $g_i\in R$ and $g\in G$, then $Hg_ig=Hg_j$ for some $g_j\in R$; in other words,  $g_i g g_j^{-1} \in H$. 
Then for $m\in M$,  $(m\otimes g_i)\cdot g=mg_i g g_j^{-1}\otimes g_j$.

One useful consequence  is that if $S$ is a commutative ring with unity with a right $H$--action (and hence an $(S,\zz[H])$--bimodule), then $\Ind_H^G(S)$  is a free left $S$--module with basis $\{1\otimes g_i\}$.

Another consequence is a naturality property of induced modules.  To describe it, note that given a group homomorphism $h \co A \to B$ and a right $\zz[B]$--module $M$, there is a pulled back  $\zz[A]$--module structure on $M$, which we denote  $M^h$, given by $m   \cdot   a=m \cdot h(a)$. 
In all of our examples, we have subgroups $H\subset G$ and $\pi_p \subset \pi$, and a  commutative diagram (with inclusions for the horizontal maps)

$$ \begin{diagram}\dgARROWLENGTH=2em
\node{\pi_p}\arrow{e,t}{ } \arrow{s,l}{\phi'}\node{\pi}\arrow{s,l}{\phi}\\
\node{H}\arrow{e,t}{ } \node{G}
\end{diagram}
 $$
   In this setting, if  $S$   is a commutative ring  with unity and $M$ is an $(S,\zz[H])$--bimodule, the naturality property of induced modules is expressed as follows.  

  \begin{prop} \label{square}If $\phi$ is surjective  
  and $\pi_p = \phi^{-1}(H)$, then $\big(\Ind_{H}^{G} (M)\big)^\phi \cong \Ind_{\pi_p}^{\pi} (M^{\phi'})$ as $(S,\zz[\pi])$--bimodules.
   \end{prop}
  \begin{proof}  The hypotheses imply that $\pi_p\backslash \pi$ and $ H\backslash G$ are in bijective correspondence. If $\gamma_i\in \pi$ satisfy $\phi(\gamma_i)=g_i$, then the discussion above  implies that 
  $$\Ind_H^G (M)=\bigoplus _{[g_i]\in H\backslash G}M\otimes g_i=
  \bigoplus _{[\gamma_i]\in \pi_p\backslash \pi}M\otimes \gamma_i=\Ind_{\pi_p}^\pi(M).$$
  The fact that the kernels of $\phi$ and $\phi'$ coincide implies that the identification is well defined; in particular it is independent of the choice of $g_i$ and their lifts $\gamma_i$.  Indeed, it is given by the  map $1\otimes \phi'\co M\otimes_{\zz[\pi_p]}\zz[\pi]\to M\otimes_{\zz[H]}\zz[G]$.
  \end{proof}
  
   \section{Example:  $M = \ff\var$}  
   Suppose that the $CW$--complex $X$ admits a surjective map $\epsilon \co \pi \to \zz$.  Let $X_p$ be the associated $p$--fold cyclic cover of $X$, with fundamental group $\pi_p=\pi_1(X_p)$.
    Fix a field $\ff$, and consider the group $\zz = \langle t \rangle$ (i.e. written multiplicatively).   Let $N$ denote the abelian group $\ff\var$ with the $(\ff\var, \zz[\zz])$--bimodule
structure  defined by $f(t) \cdot q \cdot p(t) = f(t) p(t) q$.  
  Let $p\zz \subset \zz$ denote the subgroup of index $p$, and let $\psi\co p\zz \to \zz$ denote the isomorphism $\psi(t^{pk})=t^k$.  
  
 Consider $M=N^{\psi \circ \epsilon}$, that is $\ff\var$ with the $(\ff\var, \zz[\pi_p])$--bimodule structure where $\gamma\in \pi_p$ acts by $q\cdot \gamma= q t^{\frac{\epsilon(\gamma)}{p}}$.  Set $M'=\Ind_{\pi_p}^\pi(M)$.  
  In the following theorem we  identify $M'=\Ind_{\pi_p}^\pi (\ff\var)$ and apply  Shapiro's Lemma to interpret $H_1(X_p;M)$ in terms of $H_1(X; M')$.

 \begin{theorem} \label{alex2}  The $(\ff\var,\zz[\pi])$--bimodule $M'=\Ind_{\pi_p}^\pi( \ff\var)$ is isomorphic to $( \ff\var )^p$ with the left $\ff\var$--action given by multiplication and the right $\zz[\pi]$--action given by 
 $$v\cdot \gamma= vA^{\epsilon(\gamma)} \text{ for } v\in ( \ff\var )^p, \gamma\in \pi$$
 where $A$ denotes the matrix   
\begin{equation*} 
A =  \begin{pmatrix}
 0&1& \cdots &0\\ 
\vdots&\vdots & \ddots&  \vdots\\
 0&0& \cdots &1\\
t&0& \cdots &0
\end{pmatrix}.
\end{equation*}
 Hence $H_1(X_p;\ff\var)\cong H_1(X;(\ff\var)^p)$ as left $\ff\var$ modules.
 \end{theorem} 
  \begin{proof}
  Applying Proposition~\ref{square} to the 
  subgroup $H=p\zz$ of $G=\zz$ 
shows that $M'=\Ind_{\pi_p} ^\pi (N^{(\epsilon\circ \psi)}) =  \Ind_{\pi_p} ^\pi \left( (N^\psi) ^\epsilon\right) =
\left( \Ind_{p\zz} ^\zz (N^\psi) \right)^\epsilon.$ 
  
  From Equation~(\ref{split}), $$ \Ind_{p\zz} ^\zz (N^\psi) =\bigoplus_{i=0}^{p-1} 
  \ff \var \otimes t^i,$$ where right multiplication by the generator $t$ for $\zz$ acts on the left $\ff\var$--module basis $\{ 1\otimes t^i \mid i=0,\dots, p-1\}$ by the matrix $A$.  The homology identification is immediate from Shapiro's Lemma.  
 \end{proof}

 As a corollary we derive the  known relationship between the Alexander polynomial of a knot and that of its $p$--fold cyclic cover.
 
   \begin{corollary}\label{corpoly}  If $X$ is the complement of a knot $K$, then the order  of $H_1(X_p; \qq\var)$  as  a $\qq\var$--module   is $\prod_{i=0}^{p-1}\Delta_K(\zeta_p^i t^{1/p})$, where $\Delta_K$ is the Alexander polynomial of $K$ and $\zeta_p$ is a primitive $p$--root of unity.

\end{corollary}

\begin{proof}  To make the notation transparent, write $R=\qq\var$ and $S=\qq[\zeta_p][  t^{\pm\frac{1}{p}}  ]$.  

The inclusion $R\subset S$ of principal ideal domains induces  
 an $(S,R)$--bimodule structure on $S$.   As a right $R$--module, $S$ is free of rank $p^2-p$, and hence flat.   Thus for any left $R$--chain complex $C_*$, 
  $H_i(S\otimes_{R}C_*)\cong S\otimes_{R}H_i(C_*).$ 
In particular the left $S$--modules 
 $H_1(X; S ^p)=H_1(X; S \otimes_{R}R^p)$ 
and
 $S\otimes_{R}H_1(X;R^p)$ 
are isomorphic.  

Since $S\otimes_R(R/(d))\cong S/(d)$,   it follows by the observations above that the order of the torsion of $H_1(X;R^p)$ is sent to the order of the torsion of
$H_1(X; S^p)$ via the inclusion $R\subset S$.

  Theorem~\ref{alex2}  implies that the order of the torsion of $H_1(X_p;R)$ (that is,  the Alexander polynomial of the cover $X_p$) is equal to the order of the torsion of $H_1(X;R^p)$, and so is sent to the order of the torsion of $H_1(X;S^p)$ via $R\subset S$.

The matrix $A$ is conjugate over $S$  to the diagonal matrix with entries $\zeta_p^i t^{\frac{1}{p}}$ on the diagonal.    Hence 
$$H_1(X;S^p)\cong \bigoplus_{i=1}^p H_1(X;S)$$
where the 1--dimensional action on the $i$th  summand is defined by having the meridian act as multiplication by $\zeta_p^i t^{\frac{1}{p}}$.

Since the 1--dimensional representation sending the meridian to the matrix $ (t)$ defines the Alexander polynomial, the order of the torsion of the $i$th summand is $\Delta_K(\zeta_p^i t^{\frac{1}{p}})$. Therefore the order of the sum (and hence of $H_1(X;S^p)$) is the product of the $\Delta_K(\zeta_p^i t^{\frac{1}{p}})$, as desired.
\end{proof}


\section{Structure of finitely generated  $\ff_q [\zz_p]$--modules and metabelian representations.}\label{sectionstructure}

In the remainder of the paper, we will denote the group $\zz/p\zz$ by $ \zz_p$.  We will denote by $\ff_q$ the field with $q$ elements.   In this section, we  examine the structure of the $\ff_q$ vector spaces $V$ with $\zz_p$ actions, that is, $\ff_q[\zz_p]$--modules. Fix  $p,q$   distinct prime positive integers.   Given   $f  \in \ff_q[\zz_p]$, let $R_f$ denote the quotient  of $\ff_q[\zz_p]$ by the principal ideal generated by $f$. Typically we write $f\in \ff_q[\zz_p]$ as a polynomial in $x$ which divides $x^p-1$.

\begin{prop}\label{prop4.4}  Let $k\ge 1$ and suppose that $p$ divides $q^{n}-1$ but does not divide $q^{k}-1$ for $k<n$. Then 
in $\ff_q[x]$,
$$x^p-1=(x-1)  \prod_{k=1}^{(p-1)/n} f_i(x)$$
where each $f_i$ has degree $n$ and is irreducible over $\ff_q$.  Moreover, the $f_i$ are  relatively prime (and relatively prime to $x-1$).
\end{prop}
 \begin{proof}  The group of units in $\ff_{q^k}$ is an abelian group of order $q^k-1$.  Thus $\ff_{q^n}$ contains   exactly $p-1$ primitive $p$--roots of unity. However, $\ff_{q^{n-1}}$ contains no nontrivial $p$--roots of unity.  Thus, $\ff_{q^n}$ is the splitting field for $x^p -1$.  
 
 Given this, each primitive $p$--root of unity satisfies an irreducible polynomial over $\ff_q$ of degree exactly $n$.  This yields the desired factorization, though at this point the $f_i$ are not clearly distinct.  However, if $x^p -1$ had a factor with multiplicity greater than one,  $x^p-1$ and its derivative would have a common factor.  
 \end{proof}
 
We let $\ell=(p-1)/n$  and denote by $f_1,f_2,\cdots,f_\ell$ the irreducible factors over $\ff_q$ of $ 1+x+\cdots+x^{p-1}$ and  let $f_0=x-1$.  Since $\ff_q[x]$ is a principal ideal domain,
 \begin{equation}\label{irred}
\ff_q[\zz_p]\cong R_{f_0}\oplus\cdots\oplus R_{f_\ell}.
\end{equation}
Replacing $x$ by $x^{-1}$ preserves $x^{p}-1$ up to powers of $x$, so each $f_i$ is either symmetric, $f_i(x)= f_i(x^{-1})$ up to a unit, or else has a conjugate $f_j$, $j \ne i$, so that $f_j(x)=f_i(x^{-1})$ up to a unit.
 
Every finitely generated $\ff_q[\zz_p]$--module $V$ has a canonical decomposition into its $f_i$--primary parts of the form
\begin{equation*}\label{decompofM} V=V_{f_0}\oplus V_{f_1}\oplus\cdots \oplus V_{f_\ell}
\end{equation*}
where  $V_{f_i}= \{v \in V\ | \ f_i v=0\}$. In particular, there exist natural projections $V\to V_{f_i}$ for each $i$. Each summand is isomorphic to direct sum of copies of 
 $   R_{f_i} $.

 \medskip

If $V$ and $V'$ are $\ff_q[\zz_p]$--modules, let $\Hom_{\ff_q[\zz_p]}(V,V')$ denote the set of  $\zz_p$--equivariant $\ff_q$--vector space homomorphisms from $V$ to $V'$, that is,  the $\ff_q[\zz_p]$ module homomorphisms.  
 
We now recall Schur's lemma~\cite{dummit} in the present context. 
 \begin{lemma}\label{schur}   There are isomorphisms $$\Hom_{\ff_q[\zz_p]}(R_{f_i},R_{f_j})\cong \left\{ \begin{array}{ll} 0 & \mbox{ if } i\neq j\\
  R_{f_i} & \mbox{ if } i=j .\end{array} \right.  $$
 Elements of $\Hom_{\ff_q[\zz_p]}(R_{f_i},R_{f_i})$ are expressed as multiplication by elements of $\ff_q[\zz_p]$.  In particular, an    $\ff_q$ basis for $\Hom_{\ff_q[\zz_p]}(R_{f_i},R_{f_i})$ is given by multiplication by $1,x,x^2,\cdots,x^{n-1}$ where $n=\deg(f_i)$. \qed

 \end{lemma}

\subsection{Semi-direct products}

 Let $V$ be a $\ff_q[\zz_p]$--module. We view  $\zz_p $ as a multiplicative group, generated by an element $x$ and  denote the action of $x^i \in \zz_p$ on $v\in V$ by $    x^i\cdot v$.
The semi-direct product $\zz_p \ltimes V $ is the set of  pairs $(x^i,v )$  with   multiplication   given by $$(x^i,v)( x^j, w) = (x^{i+j}, x^{-j} \cdot v + w    ).$$   

Note that $V$ inherits a $\zz$--action from the reduction map $\zz \to \zz_p$, so we can also form the semi-direct product $\zz \ltimes V$.  The subgroup $(p\zz) \ltimes V$ is isomorphic to a product, but we will continue to write it as a semi-direct product to highlight that it is a subgroup of $\zz \ltimes V$.
Note that $(x^i, v) = (1, x^i \cdot v) (x^i, 0).$ 
 
 \bigskip 
 
Take $X,X_p$, $\pi=\pi_1(X)$, $\pi_p=\pi_1(X_p)$, $\epsilon\co \pi\to \zz$ as in Section~\ref{section2}.  Fix an $m\in \pi$ satisfying $\epsilon(m)=1$. Then  conjugation by $m$ induces an automorphism of $\pi_p$. This automorphism in turn  induces an order $p$ automorphism of $H_1(X_p)$ which coincides with the action of the corresponding covering transformation.  

Given a left $\ff_q[\zz_p]$--module $V$, any (group) homomorphism $\rho\co \pi_p\to V$  factors through $H_1(X_p)$.  Call such a homomorphism  {\em equivariant}  provided $\rho(m\gamma m^{-1})=x\cdot \rho(\gamma)$.

Fix a $\ff_q[\zz_p]$--module $V$ with no nonzero elements fixed by $x$, that is,  if $V_{f_0}=0$.   In our applications we will take $V=H_1(X_p;\zz_q)_{f_i}$ for some $i>0$, and $\rho\co \pi_p\to H_1(X_p;\zz_q)_{f_i}$ the composite of the Hurewicz map $h \co \pi_p\to H_1(X_p;\zz_q)$ and the projection $H_1(X_p;\zz_q)\to H_1(X_p;\zz_q)_{f_i}$ to the $f_i$-primary component.

Then any  equivariant homomorphism $\rho\co \pi_p\to V$ satisfies $\rho(m^p)=0$, since  $\rho(m^p)=\rho(m m^p m^{-1})=x\cdot \rho(m^p)$.  Then $\rho$ extends to homomorphisms  $\pi  \to\zz\ltimes V$ (resp. $\pi\to \zz_p\ltimes V$)
and the formula 
\begin{equation}\label{extension}\tilde{\rho}(\gamma)=(x^{\epsilon(\gamma)},\rho(m^{-\epsilon(\gamma)}\gamma)),\end{equation} where $x$ denotes   the generator of $\zz$ (resp. of $\zz_p$). This extension satisfies   
$\tilde{\rho}(m)=x$, and is the unique extension of $\rho$    with this property. We will use the notation $\tilde{\rho}$ in either case depending on context. Notice that the second is obtained from the first by reducing the first factor modulo $p$.

Summarizing: 
\begin{prop}\label{propononetoone} If $V$ has no fixed vectors,  Formula {\rm (\ref{extension})} defines a one-to-one correspondence between equivariant homomorphisms $\rho\co \pi_p\to V$ and homomorphisms  
$\tilde{\rho} \co \pi  \to\zz\ltimes V$ (resp. $\tilde{\rho} \co \pi  \to\zz_p\ltimes V$) satisfying $\tilde{\rho}(m)=x$.  The restriction $\pi_p\to p\zz\ltimes V=p\zz\times V$ of $\tilde{\rho}\co \pi\to \zz\ltimes V$ to $\pi_p$  coincides with $\epsilon\times \rho$.
\qed
\end{prop}

In the case when $X$ is a knot complement $S^3-K$ and $X_p$ its cyclic cover, we take $m\in \pi$  a meridian.  Then   a homomorphism $\rho\co\pi_p\to V$ satisfies $\rho(m^p)=0$ if and only if $\rho$ 
 factors through a map $\hat{\rho} \co \pi_1(B_p) \to V$ where $B_p$ denotes the $p$--fold branched cover of $K$.   
 
The first homology  $H_1(X_p;\zz_q)$ decomposes as the sum $\zz_q\oplus H_1(B_p;\zz_q)$  where the first summand is precisely the fixed submodule, that is,  $\zz_q=H_1(X_p;\zz_q)_{f_0}$.  Using Proposition~\ref{propononetoone}, the Hurewicz map $h\co \pi_p\to H_1(B_p;\zz_q)$ determines the   homomorphisms
$$\pi\to  \zz\ltimes H_1(B_p;\zz_q) \text{ and } \pi\to \zz_p\ltimes H_1(B_p;\zz_q);$$
the first of these restricts to $\epsilon\times h\co \pi_p\to p\zz\times H_1(B_p;\zz_q)$. Composing with the projections $H_1(B_p;\zz_q)\to H_1(B_p;\zz_q)_{f_i}=H_1(X_p;\zz_q)_{f_i}$ for some $i> 0$ yields the homomorphisms 
$$\pi\to  \zz\ltimes H_1(B_p;\zz_q)_{f_i} \text{ and } \pi\to \zz_p\ltimes H_1(B_p;\zz_q)_{f_i}.$$
If $V=R_{f_i}$ then every equivariant homomorphism $\rho\co H_1(X_p)\to V$  factors through $H_1(X_p;\zz_q)_{f_i}$ and so the corresponding homomorphism $\tilde{\rho}\co\pi\to \zz\ltimes V$ factors through $\zz\ltimes H_1(B_p;\zz_q)_{f_i}.$  For general $V$ one applies Lemma~\ref{schur} to reduce to this special case.  \medskip 
 
In practice, such homomorphisms $\tilde{\rho}\co \pi\to \zz\ltimes V$ are constructed in terms of Wirtinger generators $\{x_1,\cdots, x_n\}$ of a knot group by setting $\tilde{\rho}(x_i) = (x, v_i)$ (with e.g. $m=x_n$ so $v_n=0$) and checking that the Wirtinger relations are satisfied.  The Wirtinger relation $x_ix_jx_i^{-1}x_k^{-1}$  imposes the linear equation 
\begin{equation}\label{repeqs}
(1-x)\cdot v_i + x\cdot v_j - v_k = 0
\end{equation} on the $v_i\in V$.

\section{ Example: The homology group $H_1(X_p ; \qq[\zeta_p])$)}

 Suppose that $\chi \co V \to \zz_q$ is given.  Then $\chi $ and $0\times \chi:p\zz \times V \to \zz_q$  endow  $\qq[\zeta_q]$ with   $(\qq[\zeta_q], \zz[V])$-- and $(\qq[\zeta_q],\zz [p\zz \times V])$--bimodule structures:  the left action is multiplication and the right $\zz[ V]$--action is given by $\alpha  \cdot v= \alpha\zeta_q^{\chi(v)} $ for $\alpha \in  \qq[\zeta_q]$ and $ v \in V$.  The right $\zz[p\zz \ltimes V]$--action factors through projection to $\zz[V]$.

 \begin{theorem}\label{omegarep}  $\Ind_{p\zz \ltimes V} ^{  \zz    \ltimes V } ( \qq[\zeta_q])$ and $\Ind_V^{  \zz_p    \ltimes V } ( \qq[\zeta_q])$ are both  isomorphic to $\qq[\zeta_q]^p$ as  $\qq[\zeta_q]$--vector spaces.  The right action of $\zz  \ltimes V$, respectively $\zz_p  \ltimes V$, is given via the homomorphism $\tau_\chi \co  \zz   \ltimes V \to GL_p(\qq[\zeta_q])$, defined as follows.  
 
\begin{equation*}\label{A}
\tau_\chi(x,v)=\begin{pmatrix}
 0&1& \cdots &0\\ 
\vdots&\vdots & \ddots&  \vdots\\
 0&0& \cdots &1\\
1&0& \cdots &0
\end{pmatrix}
\begin{pmatrix}\zeta_q^{\chi(v)}&0&\cdots&0\\ 
0&\zeta_q^{\chi(x\cdot v)}&\cdots&0\\ 
\vdots&\vdots&\ddots&\vdots\\
0&0&\cdots&\zeta_q^{\chi(x^{p-1}\cdot v)}
\end{pmatrix}.
\end{equation*}
  Here $x$ denotes the generator of $\zz$, respectively its image in $\zz_p$.

 \end{theorem}

 \begin{proof} As explained above, since $1,x,x^2,\cdots, x^{p-1}$ form a complete set of coset representatives for the subgroup $p\zz \ltimes V \subset \zz \ltimes V$,   the $(\qq[\zeta_q],\zz[\zz \ltimes V])$--bimodule $\Ind_{p\zz \ltimes V}^{  \zz    \ltimes V }  \qq[\zeta_q]$ is a free left $\qq[\zeta_q]$--module on the basis
 $1\otimes1,1\otimes x,\cdots, 1\otimes x^{p-1}$.    The right action via multiplication by $x$ is clear.  Multiplying by $v \in V$ we have 
 $$ (1 \otimes x^i )\cdot v   = (1  \otimes x^i v) = (1 \otimes (x^i \cdot v) x^i ) =   (1\cdot (x^i \cdot v))\otimes   x^i   = \zeta_q^{\chi(x^i \cdot v)} \otimes x^i.$$  
 
 The case corresponding to the inclusion of  $V\subset \zz_p\ltimes V$ has the same proof. 
 \end{proof}
 
Given an equivariant map $\rho\co\pi_p\to V$ as above, the extension $\tilde{\rho}\co\pi\to\zz\ltimes V$ defined above determines a diagram
 $$ \begin{diagram}\dgARROWLENGTH=2em
\node{\pi_p}\arrow{e,t}{} \arrow{s,l}{\epsilon \times \rho}\node{\pi}\arrow{s,l}{\tilde\rho}\\
\node{p\zz \ltimes V}\arrow{e,t}{} \node{\zz \ltimes V}
\end{diagram}
 $$
Applying Theorem~\ref{omegarep}, Proposition~\ref{square},  and Shapiro's Lemma we conclude the following.

\begin{theorem} With $\chi:V\to\zz_q$ as above,  the $\qq[\zeta_q]$--module $H_1(X_p; \qq[\zeta_q])$ is isomorphic to $H_1(X; (\qq[\zeta_q])^p)$ where the representation $\pi\to GL_p(\qq[\zeta_q])$ is given by composing $\tilde{\rho}$ with the representation given in Theorem~\ref{omegarep}.

\end{theorem}

\section{$\qq[\zeta_q]\var$ representations}\label{weturn}

We turn now to the set-up relevant to our knot slicing applications. Here are the ingredients:

\begin{itemize}

\item Two distinct, positive primes $p,q$, 

\item A $CW$--complex $X$ with $\pi=\pi_1(X)$, a surjection $\epsilon\co\pi\to \zz$, $X_p\to X$ the corresponding $p$--fold covering space with fundamental group $\pi_p$, and $\epsilon' \co \pi_p \to \zz$ the corresponding surjection,
\item An irreducible $\ff_q[\zz_p]$--module $V$ and a $\zz_q$--vector space homomorphism $\chi:V\to \zz_q$, 
\item A  choice of loop $m\in\pi_1(X)$ satisfying $\epsilon(m)=1$, and 
\item A nonzero equivariant  homomorphism $\rho\co\pi_p\to V$ (if $V$ is the trivial 1-dimensional module we add the requirement that $\rho(m^p)=0$).

\end{itemize}

The requirement that $V$ be irreducible is equivalent to saying that $V$ is isomorphic to $R_{f_i}$ for one of the summands in (\ref{irred}). Since $V$ is abelian, $\rho$ factors through the $\ff_q[\zz_p]$--module $H_1(X_p;\zz_q)$, and Lemma~\ref{schur} then implies that $\rho$ is surjective, since we assumed that $\rho$ is nontrivial.   Equation (\ref{extension}) defines the unique extension of $\epsilon \times \rho$ to $\tilde{\rho}\co \pi\to \zz \ltimes V$ satisfying $\tilde{\rho}(m)=x$.

In the case when $X$ is a knot complement, the condition $\rho(m^p)=0$ implies that $\rho$ factors through an equivariant homomorphism $\hat{\rho}\co\pi_1(B_p)\to V$ with $B_p$ the cyclic branched cover.

Via   $\psi \times \chi\co p\zz \ltimes V \to \zz \times \zz_q$, $\qq[\zeta_q]\var$ is a  $(\qq[\zeta_q]\var,\zz[p\zz\ltimes V ])$--bimodule. The right action of $(x^{pk},v)\in p\zz \ltimes V$ is  multiplication by $\zeta_q^{\chi(v)}t^k$:
\begin{equation}\label{action1}
f\cdot (x^{pk},v)=\zeta_q^{\chi(v)}t^k f \text{ for } f\in \qq[\zeta]\var.
\end{equation}    Proposition~\ref{square}  and Shapiro's Lemma then gives an identification 
$$H_1(X_p;(\qq[\zeta_q]\var)^{\epsilon \times \rho} )\cong H_1(X;(\Ind_{p\zz \ltimes V}^{\zz \ltimes V}\qq[\zeta_q]\var)^{\tilde \rho}). $$ 

For simplicity, we denote the induced bimodule by $(\qq[\zeta_q]\var)^p$.  Since 
$\{ (x^i, 0) \mid i=0, \dots, p-1 \}$ form a complete set of coset representatives for the subgroup $p\zz \ltimes V \subset \zz \ltimes V$, a  basis for the induced bimodule as a left $\qq[\zeta_q]\var$--module is $\{ 1\otimes (1,0), 1\otimes (x,0), \cdots, 1\otimes (x^{p-1},0)  \}$.  As explained after Equation (\ref{split}), the right action of $(x^j,v)$ on the basis element $1\otimes(x^i,0)$ is described as follows. Choose integers $k, \ell$ with $0\leq \ell <p$ and $i+j=\ell+pk$. Then 
\begin{eqnarray*}  \big(1\otimes(x^i,0)\big)\cdot (x^j,v)&=& \big(1\cdot(x^i,0)(x^j,v)(x^{-\ell},0)\big)\otimes(x^\ell,0)\\
&=&\big(1\cdot(x^{i+j-\ell},x^\ell\cdot v)\big)\otimes (x^\ell,0)\\
&=&1\cdot(x^{pk},  x^\ell\cdot v) \otimes (x^\ell, 0)\\
&=& \zeta_q^{\chi(x^{i+j}\cdot v)}t^k\otimes (x^\ell, 0)\\
&=& \zeta_q^{\chi(x^{i+j}\cdot v)}t^k\cdot \big(1\otimes (x^\ell, 0)\big)
\end{eqnarray*}
Hence the induced right action of $(x^j, v)=(x^j,0)(1,v)=(x,0)^j(1,v)\in \zz \ltimes V$ on $(\qq[\zeta_q]\var)^p$ is given 
by the matrix:
 \begin{equation} \label{bigrep1}
 \begin{pmatrix}
 0&1& \cdots &0\\ 
\vdots&\vdots & \ddots&  \vdots\\
 0&0& \cdots &1\\
t&0& \cdots &0
\end{pmatrix}^{j}
 \begin{pmatrix}\zeta_q^{\chi(v)}&0&\cdots&0\\ 0&\zeta_q^{\chi(x\cdot v)}&\cdots&0\\ 
\vdots&\vdots&\ddots&\vdots\\
0&0&\cdots&\zeta_q^{\chi(x^{p-1}\cdot v)}
\end{pmatrix}
    \end{equation}

In summary, 

\begin{theorem} \label{bigrep}  Fix an equivariant $\rho\co\pi_p\to V$   and a character $\chi\co V\to \zz_q$. With the local coefficients defined by the actions of  Equations (\ref{action1}) and (\ref{bigrep1}) and the homomorphisms $\epsilon\times \rho\co\pi_p\to p\zz \ltimes V$ and  $\tilde\rho\co \pi \to \zz\ltimes V$, the homology groups $H_1 (X_p; \qq[\zeta_q]\var)$ and $ H_1 (X; (\qq[\zeta_q]\var)^p )$ are isomorphic as  $ \qq[\zeta_q]\var$--modules, and hence
$$\Delta_{X_p,  \qq[\zeta_q]\var} (t) = \Delta_{X,   (\qq[\zeta_q]\var)^p  }(t).$$
\qed
 \end{theorem}

 We finish this section with the observation that $ \Delta_{X,   (\qq[\zeta_q]\var)^p  }(t)$ can be viewed  as a twisted polynomial in the sense of \cite{kl1}.  Precisely, 
 \begin{equation}\label{istwist} \Delta_{X,   (\qq[\zeta_q]\var)^p  }(t)=\Delta_{X,\epsilon\otimes\alpha}(t^{\frac{1}{p}})\end{equation}
 for some representation $\alpha:\pi\to GL_p(\qq[\zeta_q])$.  To see this, notice that the matrices
 \[ \begin{pmatrix}0&1  &\cdots&0\\  
  \vdots&\vdots&\ddots &\vdots\\
  0&0& \cdots&1\\
 t&0&  \cdots&0\\
 \end{pmatrix}\text{  \ and  \ }  t^{\frac{1}{p}}\begin{pmatrix}0&1  &\cdots&0\\  
  \vdots&\vdots& \ddots&\vdots\\
  0&0& \cdots&1\\
 1&0&  \cdots&0\\
 \end{pmatrix}
  \]
  are conjugate   by the diagonal matrix with diagonal entries 
  $1, t^{\frac{1}{p}},\cdots, t^{\frac{p-1}{p}}$. Since   diagonal matrices commute, the action  (\ref{bigrep1})   can be conjugated over $GL_p(\qq[\zeta_q][t^{\pm\frac{1}{p}}])$ to the action which takes $(x^j,v)$ to 
  \begin{equation} \label{bigrep2}
t^{\frac{j}{p}}  \begin{pmatrix}0&1 &\cdots&0\\ 
 \vdots&\vdots &\ddots&\vdots\\
 0&0 &\cdots&1\\
1&0&  \cdots&0\\
\end{pmatrix}^{j}
 \begin{pmatrix}\zeta_q^{\chi(v)}&0&\cdots&0\\ 0&\zeta_q^{\chi(x\cdot v)}&\cdots&0\\  
\vdots&\vdots&\ddots&\vdots\\
0&0&\cdots&\zeta_q^{\chi(x^{p-1}\cdot v)}
\end{pmatrix}
    \end{equation} 
 Replacing $t^{\frac{1}{p}}$ by $t$ yields an action of the form $\epsilon\otimes \alpha$, where 
 $\alpha \co \pi \to GL_p(\qq[\zeta_q])$ is obtained by setting $t^{\frac{1}{p}}=1$ in Equation (\ref{bigrep2}). Conjugating a representation does not change the order of the torsion of the twisted homology, and so the formula (\ref{istwist}) follows.


\section{Knot slicing obstructions}\label{obstructions}
 
 As before, $p$ and $q$ denote distinct positive prime integers, and $\zeta_q$ is a primitive $q$th root of unity.  The ring $\qq[\zeta_q]\var$ admits an involution $\bar{ }\ \co \qq[\zeta_q]\var\to\qq[\zeta_q]\var$ which sends $t$ to $t^{- 1}$ and $\zeta_q^i$ to $\zeta_q^{q-i}$.
 
 \medskip
 
 We recall how twisted Alexander polynomials obstruct the slicing of knots. Given $K\subset S^3$   an oriented knot, as above we denote by $X=S^3-K$ its complement  and $\pi=\pi_1(X)$ the knot group. The orientation of $S^3$ and $K$  uniquely determine a surjection $\epsilon\co\pi_1(X)\to \zz$. An oriented meridian $m\in\pi$ for $K$ satisfies $\epsilon(m)=1$.

Let  $X_p\to X$ be the $p$--fold  cyclic cover, with fundamental group $\pi_p$,  and  $B_p\to S^3$ the $p$--fold branched cover. Then $H_1(B_p)$ is a finite abelian  group.   Denote by $\epsilon'\co\pi_p\to \zz$ the corresponding surjection.

The {\em linking form} is a nonsingular form
$$\lk\co H_1(B_p)\times H_1(B_p)\to \qq/\zz.$$
An {\em invariant metabolizer for $\lk$} is a subgroup $A\subset H_1(B_p)$ invariant under the action of the covering transformations for which $A=A^\perp$, where $A^\perp$ denotes the perpendicular subgroup to $A$ with respect to $\lk$. Any metabolizer $A$ has order the square root of the order of $H_1(B_p)$.

Since we are assuming $p$ is prime, the homology group $H_1(B_p)$ is a torsion group. It therefore has a primary decomposition
$$H_1(B_p)=\bigoplus_{q\text{ prime }} H_1(B_p)_{(q)}$$  where $\zz_{(q)}$ is $\zz$ localized at $q$ (that is, with all primes other than $q$ inverted)
and $H_1(B_p)_{(q)}=H_1(B_p)\otimes \zz_{(q)}$.  
The linking pairing between different $q$-primary components is zero.  
Any metabolizer $A\subset H_1(B_p)$ will similarly decompose into $A=\bigoplus_{q\text{ prime}} A_{(q)}$ with $A_{(q)} \subset H_1 (B_p)_{(q)}$.  From order considerations it is clear that $A_{(q)}$ is a metabolizer for $H_1 (B_p)_{(q)}$ with respect to the restricted intersection pairing, which 
defines a nondegenerate pairing
$$\lk\co H_1(B_p)_{(q)}\times H_1(B_p)_{(q)}\to \zz[\tfrac{1}{q}]/\zz . $$

 Any $\zz_q$--character on $H_1(B_p)$ annihilates all the other primary subgroups, so we can view it as a character on $H_1(B_p)_{(q)}$.  In addition, it factors through the map $H_1(B_p) \to H_1(B_p,\zz_q)$.   Thus the set of $\zz_q$--characters on $H_1(B_p)$ corresponds bijectively to $\Hom(H_1(B_p ; \zz_q), \zz_q) \cong H^1 (B_p; \zz_q)$.

In~\cite{kl1}, building on the ideas of Casson and Gordon~\cite{cg1}, the following result (adapted to the notation of the current article) was shown.

\begin{theorem}\label{slicethm} If $K$ is slice and $p,q$ are distinct  primes, with $q\ne2$, then  there exists an invariant metabolizer $A\subset H_1(B_p)$ so that for any $\chi\in H^1(B_p;\zz_q)$ which vanishes on $A$, the twisted polynomial $\Delta_{ X_p,  \qq[\zeta_q]\var}(t)$  factors as $\lambda t^k f(t)\overline{f(t)}(t-1)^e$ for some $\lambda\in \qq[\zeta_q], k\in \zz, $ and $f(t)\in \qq[\zeta_q][t,t^{-1}]$. Here $e=1$ if $\chi$ is nonzero, and $e=0$ if $\chi$ is zero. \qed
\end{theorem}

\noindent{\bf Remark.} In that article, Theorem~\ref{slicethm} is stated  assuming $p$ and $q$ are odd, but the restriction to $p$ odd is unnecessary.

Assume $A\subset H_1 (B_p)$ is an invariant metabolizer and $\chi$ vanishes on $A$.  Then the induced map $\bar{\chi}  \co  H_1(B_p; \zz_q) \to \zz_q$ vanishes on the image $\bar{A} \subset H_1 (B_p; \zz_q)$ of $A$.   Notice that  $\bar{A}$ is a $\zz_p$--invariant subgroup of $H_1(B_p; \zz_q)$    (a $\ff_q[\zz_p]$--submodule of $H_1(B_p; \zz_q)$).  In any specific example we can enumerate the equivariant metabolizers of $H_1(B_p)$ and their images in $H_1(B_p;\zz_q)$, but the following lemma will permit us to bypass some of that work.

\begin{lemma}\label{proper}  Let $G$ be a finite abelian $q$--group with a nonsingular linking form $\lk \co G \times G \to \qq / \zz$.    Let $H$ be a metabolizer for $(G, \beta)$.  Then the inclusion
 $H\otimes \zz_q \to G \otimes \zz_q$ is not surjective. 

\end{lemma}

 \begin{proof}   
 
 We have the exact sequence of abelian groups:  $H \to G \to G/H \to 0$.  Tensoring is right exact, so we have $H\otimes \zz_q \to G \otimes \zz_q \to (G/H) \otimes \zz_q \to 0$ is exact.  The last group is nontrivial, since $G/H$ is a nontrivial $q$--group.
 
 \end{proof}

 Combining these two results we obtain the following.
 
 \begin{corollary}\label{slicethmcor} If $K$ is slice and $p, q$ are distinct  primes with $p\ne 2$, then  there exists a proper  invariant subspace $\bar{A}\subset H_1(B_p ; \zz_q)$ so that for any   $\chi\in H^1(B_p ; \zz_q)$ which vanishes on $\bar{A}$, the corresponding twisted polynomial $\Delta_{X_p,   \qq[\zeta_q]\var} (t)$  factors as $\lambda t^k f(t)\overline{f(t)}(t-1)^e$ for some $\lambda\in \qq[\zeta_q], k\in \zz, $ and $f(t)\in \qq[\zeta_q]\var$. Here $e=1$ if $\chi$ is nonzero, and $e=0$ if $\chi$ is zero. The subspace $\bar{A}$ is the reduction modulo $q$ of a metabolizer   for the linking form $\lk$. \qed
 \end{corollary}

\begin{definition} We   call a Laurent polynomial $d(t)\in\qq[\zeta_q]\var$ a {\em norm} provided $d(t)$ factors
in the form $$d(t)=\lambda t^kf(t)\overline{f(t)}.$$
\end{definition}

 Note that by multiplying by appropriate powers of $t$ one may assume that $d(t)$ and $ f(t)$ are   polynomials with nonzero constant terms, and that $k=\deg(d)/2=\deg(f)$.
 
\medskip
  
Let $K$ be an algebraically slice knot, and fix a prime number $p$.  In order to use twisted Alexander polynomials associated to the $p$--fold cover $X_p$ to show that $K$ is not slice, we must find a prime $q$ and show that, for every image $\bar A\subset H_1(B_p; \zz_q)$ of an invariant metabolizer, there is a nontrivial $\chi \co H_1(B_p; \zz_q) \to \zz_q$ which vanishes on $\bar A$ for which $\Delta_{X_p,   \qq[\zeta_q]\var} (t)/(1-t)$ is not a norm.  

\begin{definition} We call the quotient $\Delta_{X_p,   \qq[\zeta_q]\var} (t)/(1-t)^e$ the {\em  reduced twisted Alexander polynomial}   and  denote it by  $ \tD_{X_p,    \qq[\zeta_q]\var} (t)$ or $\tD_{X,    (\qq[\zeta_q]\var)^p} (t)$. Here $e=1$ if $\chi$ is nonzero, and $e=0$ if $\chi$ is zero. 
\end{definition}

Determining whether a polynomial $d(t)\in\qq[\zeta_q]\var$ is  not a norm can be a challenge in general.  However, the following observation and 
 number theoretic lemma  provide two tools which are sufficient  to deal with all the examples we calculate below. 
 
 First, if $d(t)$ is a norm, then its image in $\cc[t]$ (mapping $\zeta_q$ to $e^{2\pi i/q}$)
 factors similarly. Since $\overline{(t-z)}=(t^{-1}-\bar{z})=-\bar{z}t^{-1}(t-1/\bar{z})$ it follows that 
 the complex roots of $d(t)$ come in pairs of the form $z$ and $1/\bar{z}$.
 
 A more sophisticated  method is the following. All of the  polynomials $d(t)$ we calculate have coefficients in the subring $\zz[\zeta_q]\subset \qq[\zeta_q]$. Although $\zz[\zeta_q]$ is not a unique factorization domain for most primes $q$, we can nevertheless apply the following version  of Gauss's lemma.
 
 \begin{lemma}\label{gausslem} Let $q, r$ be primes  and suppose $r=nq +1$ for some positive integer $n$. Choose $b  \in \zz_r$  so that $b\ne 1$ and $b^q=1$, and let $\phi\co\zz[\zeta_q]\to \zz_r$ be the ring homomorphism sending $1$ to $1$ and $\zeta_q$ to $b$. 
 
 Let
  $d(t)\in \zz[\zeta_q][t]$ be a   polynomial of degree $2k$.  Assume its image $\phi(d(t))\in   \zz_r [t]$ also has degree $2k$.  
  If $d(t)$ is a norm (over $ \qq[\zeta_q]$)   then $\phi(d(t))\in  \zz_r [t]$ factors as the product of two polynomials of degree $k$.
 \end{lemma}
 \begin{proof}  Let $\kappa=\ker \phi$. This is a maximal ideal (since the quotient is a field)  in the Dedekind domain $\zz[\zeta_q]$. The localization $\zz[\zeta_q]_{ \kappa }$ is therefore a discrete valuation ring and hence a unique factorization domain (\cite{dummit}).  The homomorphism $\phi$ extends to  $\zz[\zeta_q]_{ \kappa }$, since localizing inverts elements in the complement of $\kappa$, which are sent by $\phi$ to units in $\zz_r$.
 
Since $\phi(d(t))$ has degree $2k$, the leading coefficient of $d(t)$ does not lie in $\kappa$, and hence is a unit in $\zz[\zeta_q]_{ \kappa }$.  Gauss's lemma then implies that if $d(t)$ is a norm in $\qq[\zeta_q]$, it is the product of two degree $k$ polynomials in $\zz[\zeta_q]_{ \kappa }[t]$. Its image $\phi(d(t))$ is then a product of two polynomials, necessarily of degree $k$.  
 \end{proof}


\section{Algorithm to compute twisted polynomials from a Wirtinger presentation}\label{algo}

In our earlier article~\cite{kl1} we computed the twisted polynomials corresponding to $\rho$ by working with a CW--complex homotopy equivalent to the cover $X_p$, using the Reidemeister-Schreier process to find a CW--complex for  $X_p$ in terms of one for $X$, or, what amounts to the same thing, a presentation of $\pi_p$. This becomes unwieldy for a  knot  whose  group has a large presentation, since the number of 1--cells and 2--cells is roughly multiplied by $p$ in a $p$--fold cover.  Computing  downstairs, 
that is,  using the representation $\tilde\rho$ instead of $\rho$,  streamlines the computation and can be easily implemented using a computer algebra package such as MAPLE.

 The following  discussion explains how to compute twisted polynomials which arise in Theorem~\ref{slicethm}.  It applies to general knots in $S^3$, described in terms of a knot projection and the associated Wirtinger presentation of the knot group.

Recall that the Wirtinger presentation of $\pi$  has meridian generators $x_i, i=1,2,\cdots ,n,$ where $n$ is the number of strands in a projection of the knot.  The knot group is generated by these; in fact, up to homotopy equivalence,  
the knot complement has a CW--structure with the base point   the only 0--cell and the 1--cells  precisely the Wirtinger meridians. Moreover, there is one 2--cell for each crossing in the projection, attached using the Wirtinger relation $  x_i ^{-1}x_kx_jx_k^{-1}$. The  surjection $\epsilon\co\pi_1(X)\to \zz$  takes every meridian $x_i$ to $1$. 
Using this CW--structure,   one can compute the differentials $\partial_1$ and $\partial_2$, as follows.  

Since $X$ has only one 0--cell, the differential $\partial_1 \co C_1(\tilde{X})\to C_0(\tilde X)$  is given by the column vector with entries $x_i-1$.

The {\em Fox matrix of free partial derivatives} is an $n\times n$ matrix with coefficients in the group ring $\zz[\pi]$ which  represents the differential 
$\partial_2 \co C_2(\tilde{X})\to C_1(\tilde{X})$.   
Explicitly, the Wirtinger relation $x_i=x_jx_kx_j^{-1}$ contributes a row to the Fox matrix with $-1$ in the $i$th column, $1-x_i$ in the $j$th column, and $x_j$ in the $k$th column.
The {\em reduced Fox matrix} is defined to be the $(n-1)\times (n-1)$ matrix obtained by dropping the last  row and column.   
 
 The basic relation between the Wirtinger presentation,  its reduced Fox matrix $F$, and homology is the following. If  $R$ is a principal ideal domain and  $r\co \pi\to GL_p(R)$ is a representation of $\pi$ which sends the last meridian $x_n$ to a matrix $M$, let $r\co \zz[\pi]\to gl_p(R)$ denote the natural extension to the group ring and call $r(F)$ the {\em substituted reduced Fox matrix} for the Wirtinger presentation and representation $r$.    Then it is well-known (and proved in the present context in~\cite{kl1}) that if the determinant of $r(F)$ is nonzero, then the order of the torsion of $H_1(X;R^n)$ is equal to the determinant of $r(F)$ times a factor which depends only on $M$ and $H_0(X;R^p)$. We refer the reader to~\cite{kl1} for details, but note that computing $H_0(X;R^n)$ is a simple task. 
 
  In our context we take $r$ to be $\tilde{\rho}$ and conclude 
  $$\Delta_{X, \qq[\zeta_q]\var} = \frac{\det (r(F))}{ \det( \tilde{\rho}(x_n) - I)}(1-t)^s,    $$ where $s = 1$ if $\chi$ is trivial and $s=0$ otherwise.

 \medskip
 
Of course, one can use other presentations of the knot group  rather than the Wirtinger presentation. Certain classes of knots have more convenient presentations of their knot group, for example torus knots or pretzel knots. But with other presentations of $\pi$, more care needs to be taken, and we refer the reader to the article~\cite{kl1} where these issues are explained.

\medskip

One last observation about calculations is in order. If $\chi\co V\to\zz_q$ and $\chi'\co V\to \zz_q$ are nonzero multiples of each other, say $\chi'=n\chi$, then there is a Galois automorphism $\alpha\co \qq[\zeta_q]\to\qq[\zeta_q]$ so that
$\zeta_q^{\chi(v)}=\alpha(\zeta_q^{\chi'(v)})$ for all $v\in V$, namely $\alpha(\zeta_q^i)=\zeta_q^{ni}$.
In particular the associated twisted Alexander polynomials are Galois conjugates of one another.

As a consequence, if $V$ is 1--dimensional, the twisted polynomials associated to nontrivial representations that factor through $V$ are all Galois conjugates.  Notice that a Galois conjugate of   $d(t)\in\qq[\zeta_q]\var$ is a norm if and only if $d(t)$ itself is a norm. Notice further that the equivariant automorphisms of $V$ are given by multiplication by a nonzero scalar, and so there is a unique 
twisted polynomial (up to Galois automorphisms) associated to any  1--dimensional invariant subspace of $H_1(B_p;\zz_q)$. 

If $V \subset H_1(B_p;\zz_q)$  has dimension greater than one, the   twisted polynomials corresponding to nontrivial characters that factor through the projection to $V$  need not be Galois conjugates of one another, even if $V$ is irreducible.

\section{Examples: 12 crossing prime knots}
Among prime knots of 12 or fewer crossings, there are 175 that are algebraically slice.  See the   table of knot invariants KnotInfo~\cite{liv} for details.  Of these, 157 have been previously been shown to be topologically slice.  This was done basically by finding explicit slice disks, or else using the theorem of Freedman~\cite{fr}  which states that Alexander polynomial one knots are slice.  For 11 and 12 crossing knots, the most complete search was done by Alex Stoimenow~\cite{stoi}, with the results posted on his website.  Which of the remaining 18 knots are topologically slice has remained open for several years.    In this section we illustrate the power of Corollary~\ref{slicethmcor} by demonstrating that 16 of these remaining 18 algebraically slice knots are not slice.  We also show that one of   them, $12_{a990}$, is slice, and one, $12_{a631}$ remains a mystery.  (Of the 16, we have observed that exactly two are 2--bridge knots, and these can be shown not to be slice by a calculation based on Casson and Gordon's original work~\cite{cg1} also.) 

  The 18 knots of interest  are listed in Table~1.  Also listed are the Alexander polynomials of  these knots.  The column headed ``$p$'' indicates which cover we use to prove that the knots are not slice, and the column labeled ``$q$'' indicates what type of torsion we consider. In the column headed ``$H_1(B_p)$'' the notation $a^m b^n $ is shorthand for $(\zz_a)^m\oplus (\zz_b)^n$.

\begin{table}[h]\label{table18}
\centering
\begin{tabular}{c c c c c }
Knot & Alexander Polynomial & $p$ & $H_1(B_p)$&$q$\\
\hline
$11_{n45}$&$(2t^2-2t+1)(t^2-2t+2)$&$3$&$13^2$&$13 $\\
$11_{n145}$&$(t^3-2t^2+t+1)(t^3+t^2-2t+1)$&$3$&$13^2$&$13 $\\
$12_{a169}$&$(2t^2-3t+2)^2$&$3$&$25^2$&$5 $  \\
$12_{a596}$&$(2t^2-4t+3)(3t^2-4t+2)$&$3$&$43^2$&$43 $ \\
$12_{a631}$&$(t-2)(2t-1)(2t^2-2t+1)(t^2-2t+2)$&&&$ $\\
$12_{a990}$&$(t^2-t+1)^2(t^2-3t+1)^2$&&&$ $\\
$12_{n31}$&$(2t-1)(t-2)$&$3$&$7^2$&$7 $\\
$12_{n132}$&$(2t^2-3t+2)^2 $&$3$&$25^2$&$5 $\\
$12_{n210}$&$(t^3-t+1)(t^3-t^2+1) $&$3 $&$7^2 $&$7 $\\
$12_{n221}$&$ (t^2-t+1)^2$&$2 $&$9 $&$ 3$  \\
$12_{n224}$&$ (2t-1)(t-2)(t^2-t+1)^2$&$3 $&$2^2 14^2 $&$7 $ \\
$12_{n264}$&$(t^2-2t+2)(2t^2-2t+1) $&$3 $&$13^2 $&$ 13$\\
$12_{n536}$&$(t^3-4t^2+3t-1)(t^3-3t^2+4t-1) $&$5 $&$11^2 $&$ 11$\\
$12_{n681}$&$(t^4-t^3+t^2-t+1)^2 $&$2 $&$25$&$5$\\
$12_{n731}$&$(t^3-3t^2+5t-2)(2t^3-5t^2+3t-1) $&$3 $&$4^2  13^2$&$13 $\\
$12_{n812}$&$(t^2-t+1)^2 $&$2 $&$9$&$3 $\\
$12_{n813}$&$(2t-1)(t-2)(t^2-t+1)^2 $&$3 $&$28^2 $&$ 7$ \\
$12_{n841}$&$(2t-1)(t-2)(t^2-t+1)^2 $&$3 $&$28^2 $&$ 7$\\
\hline
\end{tabular}
\vskip.2in\caption{}
\end{table}
 
 \vskip.1in
   Given a knot $K \subset S^3$, we apply twisted polynomials by letting $X$ denote $S^3 - K$, $X_p$ its $p$--fold cyclic cover, and $B_p$ its $p$--fold cyclic branched cover.  A choice of $\ff_q[\zz_p]$--module $V$, a character $\chi \co V \to \zz_q$, and an equivariant homomorphism $\rho \co H_1(B_p; \zz_q) \to V$ determines a $(\qq[\zeta_q]\var , \zz[\pi])$--module structure on $(\qq[\zeta_q] \var)^p $, as in Section~\ref{weturn}. The resulting reduced twisted Alexander polynomial is denoted  $ \tD_{X,   (\qq[\zeta_q]\var)^p} (t)$. 
 \vskip.1in
 
 \subsection{The knot $\mathbf {12_{a990}}$ is slice.}
 In Figure~\ref{12a990} we illustrate the connected sum of $12_{a990}$ with right- and left-handed trefoils.  Since the sum of these trefoils is slice, forming the connected sum does not change  the concordance class.  If the two band moves are made along the indicated arcs, the resulting three component link is an unlink.  Thus, the connected sum is slice, as desired. This construction was inspired by a similar one developed by Tamulis~\cite{tam}.

 \subsection{The knot $\mathbf {12_{a169}}$}  We will show that the 12 crossing alternating knot $K=12_{a169}$  is not slice.  This example exhibits all the phenomena discussed in the previous sections, and in particular $12_{a169}$ is determined not to be slice by the calculation of a single twisted Alexander polynomial.

  \begin{figure}\label{12a990}
  \fig{.5}{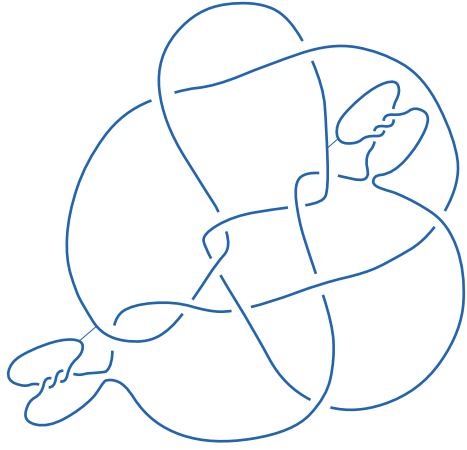}
  \caption{The knot $12_{a990} \# T_{2,3} \# T_{2,-3}$}
  \end{figure}

Let  $B_3 $ be the $3$--fold branched cover of $S^3$ branched over $K = 12_{a169}$.
Then a standard calculation using   the Seifert form (see~\cite{rolf}) shows that $H_1(B_3)=\zz_{25}\oplus \zz_{25}$.   We take $q=5$.  Then $H_1(B_3;\zz_5)=\zz_5\oplus\zz_5$, and, since  $p$ does not divide $q-1$, Proposition~\ref{prop4.4}   implies that, as an $\ff_5[\zz_{3}]$--module, 
$H_1(B_3;\zz_5)$ must be isomorphic to the the irreducible module $  R_{1+x+x^2}.$

We have the canonical homomorphism $ \rho \co \pi_3 \to H_1(B_3;\zz_5)$  and   its extension $\tilde\rho\co \pi \to \zz \ltimes H_1(B_3;\zz_5)$. 
If $A\subset H_1(B_3)$ is an invariant metabolizer, then its image $\bar A\subset H_1(B_3;\zz_5)$ is  a proper invariant subspace by  Lemma~\ref{proper}, and hence $\bar A=0$  since $H_1(B_3;\zz_5)$ is irreducible.
 
Thus every $\chi\in \Hom(H_1(B_3),\zz_5)$ vanishes on $A$ and so  to prove $K$ is not slice, it suffices to find a single $\chi$ so that   so that  the corresponding twisted polynomial
$\tD_{X,   (\qq[\zeta_q]\var)^p} (t)$ is not a norm.   The twisted polynomial $\Delta_{X,  (\qq[\zeta_q]\var)^p}(t)$ is computed using the method described in Section \ref{algo}.

 The knot group has  Wirtinger presentation with generators $x_1,x_2,\cdots, x_{12}$ and relations:
 
 $$x_1=x_8^{-1} x_2 x_8, x_2=x_7^{-1}x_3x_7, x_3=x_6^{-1}x_4x_6, x_4=x_{11}^{-1}x_5x_{11},$$
  $$x_5=x_4^{-1} x_6 x_4, x_6=x_3^{-1}x_7x_3, x_7=x_2^{-1}x_8x_2, x_8=x_{1}^{-1}x_9x_{1},$$
   $$x_9=x_{12} x_{10} x_{12}^{-1}, x_{10}=x_{9} x_{11} x_{9}^{-1}, x_{11}=x_5^{-1}x_{12}x_5, x_{12}=x_{10}x_{11}x_{10}^{-1}.$$

 One checks that the assignment $x_i\mapsto x v_i$ where
\begin{equation*}\label{eq6.3}
\begin{split}  v_1=4+2x, v_2=2+x, & v_3=0,v_4=3+4x, v_5=2+3x, v_6=4+4x, 
\\
v_7=1,v_8=3+x, & v_9=2x, v_{10}=2,  v_{11}=1, v_{12}=0  \end{split} \end{equation*}
solves the linear system given by (\ref{repeqs}), yielding   $\tilde{\rho}\co\pi\to \zz \ltimes {R_{1+x+x^2}}$.    

For $\chi\co R_{1+x+x^2}\to \zz_5$ we take the homomorphism determined by 
$$\chi(1)=1, \chi(x)=0.$$
Then the corresponding right  $\pi$--action on $ (\qq[\zeta_5]))^3$    is computed using Theorem~\ref{bigrep}. For example, the meridian $x_1$ is sent to $(x, 4+2x)\in\zz\ltimes R_{1+x+x^2}$, and so using Equation (\ref{bigrep1}) one computes that $x_1$ acts by the $3\times 3 $ matrix
\begin{equation*}
\begin{pmatrix} 0&\zeta_5^{-2}&0\\ 0&0&\zeta_5^{-2}\\ \zeta_5^4 t&0&0\end{pmatrix}
 \end{equation*}
 We have used the fact that $x^2=-1-x$ in $R_{1+x+x^2}$.

 In this way we obtain a homomorphism $\pi\to GL_3(\qq[\zeta_5]\var)$  which we apply to the entries in the reduced Fox matrix. The determinant of the resulting $33\times 33$ matrix (this is a sparse matrix: only 132 entries are nonzero, and all nonzero entries have the form $\pm1$ or $\pm t^k\zeta_5^r$) equals
 $$-t^3(4t^2+t\zeta_5^2+t\zeta_5^3+5t+4)(t-1)^2.$$
 Dividing by $(t-1) $ yields   
 $$ \Delta_{X,  (\qq[\zeta_q]\var)^p}(t) = -t^3(t-1)\big(4t^2+(\zeta_5^3+\zeta_5^2+5)t+4\big)$$
 and so (up to units in $\qq[\zeta_5]\var$), 
 $$\tD_{X,  (\qq[\zeta_q]\var)^p}(t) = 4t^2+(\zeta_5^3+\zeta_5^2+5)t+4  .$$
 
 This polynomial is not a norm. Indeed, the map $\zz[\zeta_5]\to \zz_{41}$ taking $\zeta_5$ to $10$  ($10^5=1$ in $\zz_{41}$) maps this polynomial to the irreducible polynomial $37 t^2 + 2t + 37$ in $  \zz_{41} [t]$.  Applying Lemma~\ref{gausslem} shows that $4t^2+(\zeta_5^3+\zeta_5^2+5)t+4 $ is not a norm.  It then follows from Corollary \ref{slicethmcor}
that  $12_{a169}$ is not slice.

 \subsection{The knot $\mathbf {12_{n132}}$} This knot has the same Alexander polynomial and homology of the 3--fold branched cover as the knot $12_{a169}$ which was treated in the previous subsection.  Thus we argue in precisely the same way as we did before, solving  the linear system (\ref{repeqs}) for  $v_i\in R_{1+x+x^2}$     and computing the determinant of the corresponding substituted reduced Fox matrix.  This time the calculation yields
\begin{equation*}
\begin{split}
 \tD_{X,  (\qq[\zeta_q]\var)^p}(t)  = \  & ( t-1  )   \big( 5\,{t}^{3}+ \left( -12\,{\zeta_5}^{4}-2\,{\zeta_5}^{3}
+2\,{\zeta_5}^{2}+2\,\zeta_5 \right) {t}^{2}\\ 
&+  ( 2\,{\zeta_5}^{4}+2\,{\zeta_5}^{3}-2\,{\zeta_5}^
{2}-12\,\zeta_5  ) t+5  \big). 
 \end{split}
 \end{equation*}
This polynomial is not a norm by Lemma~\ref{gausslem}, since mapping $\zz[\zeta_5]$ to $\zz_{31}$ by sending $\zeta_5$ to $2$ yields 
 $(t+30) (5 t^3+21 t^2+16 t+5) $, and the cubic term is irreducible.  Hence $12_{n132}$ is not a slice knot.

 \subsection{The knot  $\mathbf {12_{n813}}$} 
  The knot $K=12_{n813}$  has Alexander polynomial  $(2t-1)(t-2)(t^2-t+1)^2$  and the homology of the 3--fold branched cover of $K$ is $\zz_{28}\oplus \zz_{28}$. 
  
  We take $p=3$ and $q=7$. With this choice $x^p-1$ factors over $\ff_7[\zz_{3}]$  as 
 $(x-1)(x+3)(x+5)$, and hence $H_1(B_3;\zz_{7})$ splits as a $\ff_7[\zz_{3}]$--module:
 $$H_1(B_3;\zz_7)\cong R_{x+3}\oplus R_{x+5}.$$
 (nondegeneracy of the linking form requires both possible primary components to be nonzero, and hence $H_1(B_3;\zz_7)$ cannot be isomorphic to $R_{x+3}\oplus R_{x+3}$.)   Fix an isomorphism $H_1(B_3;\zz_7)\cong R_{x+3}\oplus R_{x+5}$
 
 If  $A\subset H_1(B_3)$ is an invariant metabolizer, its image   $\bar A\subset H_1(B_3;\zz_7)$ must either be   $R_{x+3}$  or $R_{x+5}$, since it is invariant and must have order 7.  If $\bar A$ equals $R_{x+3}$, then any   equivariant $\rho_5\co \pi_3\to R_{x+5}$ vanishes on $\bar A$ by Lemma~\ref{schur}.  Similarly if $\bar A$ equals $R_{x+5}$,  any   equivariant $\rho_3\co \pi_3\to R_{x+3}$ vanishes on $A$. 
 
 We construct $\rho_3$ and its extension $\tilde{\rho}_3 \co \pi \to \zz \ltimes R_{x+3}$ by solving the 
  linear system  (\ref{repeqs}).  Since $R_{x+3}$ is generated as a $\ff_7$--vector space by $1$, we can take $\chi_3\co R_{x+3}\to\zz_7$ defined by $\chi_3(1)=1$. 
 
Using the algorithm described above a calculation yields
\begin{equation*}\begin{split}\tD_{X, \tilde{\rho}_3}=  (t+1) \big(-{t}^{3}+ \left( -3\,{\zeta_7}^{4}-3\,\zeta_7-3\,{\zeta_7}^{2}-5\,{\zeta_7}^{3}-5\,{\zeta_7}^{5}-5
\,{\zeta_7}^{6} \right) {t}^{2}\\ 
 + \left( -5\,\zeta_7-5\,{\zeta_7}^{2}-3\,{\zeta_7}^{3}-5\,{\zeta_7}^{
4}-3\,{\zeta_7}^{5}-3\,{\zeta_7}^{6} \right) t-1\big).
 \end{split}\end{equation*}
Similarly one finds $\rho_5\co \pi_3\to R_{x+5}$ and $\chi_5\co R_{x+5}\to\zz_7$. The resulting polynomial is
 \begin{equation*}\begin{split}\tD_{X, \tilde{\rho}_5}=
 (t+1)\big( {t}^{3}+ \left( 5\,\zeta_7+5\,{\zeta_7}^{4}+5\,{\zeta_7}^{2}+3\,{\zeta_7}^{3}+3\,{\zeta_7}^{5}+3\,
{\zeta_7}^{6} \right) {t}^{2}\\ 
+ \left( 3\,\zeta_7+3\,{\zeta_7}^{2}+5\,{\zeta_7}^{3}+3\,{\zeta_7}^{4}+
5\,{\zeta_7}^{5}+5\,{\zeta_7}^{6} \right) t+1 \big). \end{split}\end{equation*}
 Neither of these are norms.  One way to see this is to note that
 $\overline{t+1}=t^{-1}+1=t^{-1}(1+t)$, and $-1$ is not a root of the cubic factor.  Alternatively, map to $\zz_{43}$  sending $\zeta_7$ to $4$; the result does not factor into a product of quadratics.  Hence $12_{n813}$ is not slice.

\subsection{The knot $\mathbf {12_{n841}}$} 
The    Alexander polynomial of $K= 12_{n841}$ is the same as that of $12_{n813}$, and the homology of its 3--fold branched cover is  also $\zz_{28}\oplus   \zz_{28}$, and so $H_1(B_3;\zz_7)=R_{x+3}\oplus R_{x+5}$.  We compute  in exactly the same way as for $12_{n813}$.  This time the results are
\begin{equation*}\begin{split}\tD_{X, \tilde{\rho}_3}=  
(t+1)\big(1+ \left( 5\,\zeta_7+5\,{\zeta_7}^{4}+5\,{\zeta_7}^{2}+3\,{\zeta_7}^{3}+3\,{\zeta_7}^{5}+3\,{\zeta_7}^{6}
 \right) t\\ 
 + \left( 3\,\zeta_7+3\,{\zeta_7}^{2}+5\,{\zeta_7}^{3}+3\,{\zeta_7}^{4}+5\,{\zeta_7}^{5}+5
\,{\zeta_7}^{6} \right) {t}^{2}+{t}^{3}
\big) 
\end{split}\end{equation*} and  
  \begin{equation*}\begin{split}\tD_{X, \tilde{\rho}_5}=
 (t+1)\big({t}^{3}+ \left( 5\,{\zeta_7}^{2}+5\,\zeta_7+5\,{\zeta_7}^{4}+3\,{\zeta_7}^{3}+3\,{\zeta_7}^{5}+3\,
{\zeta_7}^{6} \right) {t}^{2}\\ 
+ \left( 3\,\zeta_7+3\,{\zeta_7}^{2}+5\,{\zeta_7}^{3}+3\,{\zeta_7}^{4}+
5\,{\zeta_7}^{5}+5\,{\zeta_7}^{6} \right) t
+1\big). \end{split}\end{equation*}
 Neither of these are norm,  as one can see by mapping to $\zz_{43}$ taking $\zeta_7$ to $4$. Hence $12_{n841}$ is not slice.

\subsection{The knot $\mathbf {12_{n224}}$}
The Alexander polynomial of $12_{n224}$ is also the same as that of $12_{n813}$. The homology of its 3--fold branched cover is  slightly different than the previous two: $H_1(B_3)=\zz_2\oplus \zz_2\oplus\zz_{14}\oplus \zz_{14}$, but with $\zz_7$ coefficients we again get $\zz_7\oplus \zz_7$. Arguing as above, the polynomials $\tD$ are, for $\tilde{\rho}_3$, 
\begin{equation*}\begin{split}1+ \left( 5\,{\zeta_7}^{4}+5\,{\zeta_7}^{2}+5\,\zeta_7+{\zeta_7}^{3}+{\zeta_7}^{5}+{\zeta_7}^{6} \right) t
+6\,{t}^{2}\\ 
+ \left( {\zeta_7}^{2}+\zeta_7+{\zeta_7}^{4}+5\,{\zeta_7}^{3}+5\,{\zeta_7}^{5}+5\,{\zeta_7}^{6}
 \right) {t}^{3}+{t}^{4}
 \end{split}\end{equation*}
 and, for $\tilde{\rho}_5,$
 \begin{equation*}\begin{split}1+ \left( {\zeta_7}^{4}+{\zeta_7}^{2}+\zeta_7+5\,{\zeta_7}^{3}+5\,{\zeta_7}^{5}+5\,{\zeta_7}^{6} \right) t 
+6\,{t}^{2}\\+ \left( 5\,{\zeta_7}^{4}+5\,{\zeta_7}^{2}+5\,\zeta_7+{\zeta_7}^{3}+{\zeta_7}^{5}+{\zeta_7}^{6}
 \right) {t}^{3}+{t}^{4}.
 \end{split}\end{equation*}
  These are irreducible by Lemma~\ref{gausslem} since they map to irreducible fourth degree polynomials over 
  $\zz_{29} $ by taking $\zeta_7$ to $7$.  Hence $12_{n224}$ is not slice.

 \subsection{The knots $ \mathbf {11_{n45},\,  11_{n145},\,   12_{a596},\,  12_{n31},\,   12_{n210},\,  12_{n264},\,   12_{n731}}$} These knots are treated exactly in the same way as were $12_{n813}$, $12_{n841}$, and $12_{n224}$, using the choices of $p$ and $q$ given in Table 1. In each case the homology $H_1(B_p;\zz_q)$ splits as the sum of  two 1-dimensional subspaces $H_1(B_p;\zz_q)=R_{x-a}\oplus R_{x-b}$, where $ a$ and $b$ are the two $p$th roots of $1$ in $\zz_q$. The resulting polynomials $\tD$ are not norms and so these knots are not slice.
 
 \subsection{The knot $\mathbf {12_{n536}}$} 
 For this knot we take $p=5$ and $q=11$.   The $5$th roots of unity in $\zz_{11}$ are $3,4,5$ and $9$. One can check by direct computation (the system (\ref{repeqs}) admits no solutions for $v_i$ in $R_{x-5}$ or $R_{x-9}$) or using an observation of Hartley~\cite{hart}, that only $3$ and $4$ arise, that is,  $$H_1(B_5;\zz_{11})=R_{x-3}\oplus R_{x-4}.$$
 The rest of the calculation proceeds just as in the previous examples, and   one concludes $12_{n536}$ is not slice.

\subsection{The knot  $\mathbf {12_{n681}}$}   For this  knot   we use the  2--fold cover. 

For $K=12_{n681}$, $H_1(B_2;\zz)=\zz_{25}$. Thus $H_1(B_2;\zz_5)=R_{x+1}$ and hence by Lemma~\ref{proper} every invariant metabolizer is sent to zero, i.e. $\bar A=0$. Thus to show $K$ is not slice one need only find a single nontrivial $\chi\co R_{x+1}\to \zz_5$  so that the corresponding $\tD$ is not a norm.  For one choice   the result is
 \begin{equation*}\begin{split}  (t-1)^2  \big({t}^{4}+ \left(  \zeta_5^2+ \zeta_5^3-1 \right) {t}^{3}
 + \left( -2-\zeta_5^2-\zeta_5^3 \right) {t}^{2} 
 + \left( \zeta_5^2+ \zeta_5^3-1 \right) t +1\big). 
 \end{split}\end{equation*}
Showing this is not a norm is more challenging than the other examples. In fact Lemma~\ref{gausslem} does not help: the image of the polynomial $$p(t)={t}^{4}+ \left(  \zeta_5^2+ \zeta_5^3-1 \right) {t}^{3}
 + \left( -2-\zeta_5^2-\zeta_5^3 \right) {t}^{2} 
 + \left( \zeta_5^2+ \zeta_5^3-1 \right) t +1$$ factors as a product of quadratic polynomials in $ \zz_r [t]$ for every prime $r$ with $r\equiv 1 \mod{5}$.
 
  We instead argue as follows.  Set $\zeta_q=e^{2\pi i/5}$.  Note that $p(t)$ is real.  In fact, since $\zeta_5^2+\zeta_5^3$ satisfies $x^2+x-1$, the coefficients of $p(t)$ lie   in $\zz[\tfrac{1+\sqrt{5}}{2}]$, the ring of integers in the quadratic extension $\qq[\sqrt{5}]$ of $\qq$.  
 
 The mapping $\zz[\tfrac{1+\sqrt{5}}{2}]$ to $\zz_{19}$ taking $\tfrac{1+\sqrt{5}}{2}$ to $14$ sends  $p(t)$ to the irreducible polynomial $t^4+13t^3+3t^2+13t+1$.  Therefore, $p(t)$ is irreducible over $\zz[\tfrac{1+\sqrt{5}}{2}]$.    Since $\zz[\tfrac{1+\sqrt{5}}{2}]$ is a unique factorization domain, Gauss's Lemma implies that $p(t)$ is irreducible over $\qq[\sqrt{5}]$.  
 
 Suppose that $p(t)$ is reducible over $\qq[\zeta_q]$.  
 The Galois group of the degree 2 extension $\qq[\zeta_q]$ over $\qq[\sqrt{5}]$ is $\zz_2$, generated by complex conjugation. Since $p(t)$ is irreducible over $\qq[\sqrt{5}]$, it must factor over $\qq[\zeta_q]$ into  complex conjugate factors.  This implies that any real roots of $p(t)$ have even multiplicity.  But one can easily check that $p(s)$ has exactly  two real roots and they are distinct.  This contradiction shows that $p(t)$ is not a norm over $\qq[\zeta_q]$. 
 
\subsection{The knot $\mathbf {12_{n812}}$}
Take $K=12_{n812}$ and $p=2$ (for other $p$ the homology is either trivial or the resulting polynomials is a norm). In this case $H_1(B_2;\zz)=\zz_9$. Thus $H_1(B_2;\zz_{3})=\zz_{3}=R_{x+1}$ and every metabolizer is sent to zero. One choice of $\chi$ yields the   polynomial $\tD=(t-1)^2(3t^2+5t+3)$. This is not a norm because $(t-1)^2$ is a norm, but $ 3t^2+5t+3$ is irreducible over $\qq[\zeta_3]$, as one sees by mapping to $\zz_7$ and using Lemma~\ref{gausslem}.

\subsection{The knot $\mathbf {12_{n221}}$} For $K=12_{n221}$, we take $p=2$ and $q=3$. We have $H_1(B_2;\zz)=\zz_9$, and so $H_1(B_2;\zz_{3})=\zz_{3}=R_{x+1}$, and hence the image in $ H_1(B_2;\zz_{3})$ of any invariant metabolizer  
is trivial.   For one choice of $\chi$ the corresponding $\tD$ equals $(3t^2+5t+3)(t-1)^2.$   This is the same polynomial that appeared   for the knot $12n812$, and is not  a norm in $\qq[\zeta_3]\var$. Hence $12_{n221}$ is not slice.

 \subsection{The  knot $\mathbf {12_{a631}}$} 
There remains the one algebraically slice knot of 12 crossings or less which is not known to be slice: $12_{a631}$.

\section{Pretzel Knots}
Fox~\cite{fox} asked in 1963 whether all knots are reversible and pointed to the knot $8_{17}$ as an obviously 
eversible knot. (Fox used the word {\it invertible} but we use {\it reversible} to distinguish the operation from the concordance inverse.)  Soon after, Trotter~\cite{trot} used 3-stranded pretzel knots to show that nonreversible knots exist.  It was several years until Hartley~\cite{hart} developed techniques that permitted the determination of the reversibility of all knots with low crossing number.

In~\cite{liv1} it was first shown that there are knots that are not concordant to their reverses, using Casson-Gordon invariants.  Kearton~\cite{kear} used these examples to show that mutation acts nontrivially on concordance.  These examples were built specifically so that the Casson-Gordon method could be applied.  In~\cite{naik} techniques were developed that could be used to show that some pretzel knots and their reverses are not concordant.  It was in~\cite{kl2} that  Fox's original test case, $8_{17}$, was shown not to be concordant to its reverse.  This provided the simplest example of a knot and it mutant being distinct in concordance.  In~\cite{kl2} a 4--stranded pretzel knot was shown to be distinct from a mutant in concordance:  $P(7,2,-5,3) \ne P(7,2,3,-5)$.

In this section we demonstrate the power of the computational method
developed in Section~\ref{algo} with some further pretzel knot computations.

As a warm-up,  we show that the mutant $P(3,5,-3,-5,7)$ of the slice pretzel knot $P(3,-3,5,-5,7)$ is not slice. The homology of the 3-fold branched cover $B_3$ of $P(3,5,-3,-5,7)$ is 
$(\zz_7)^2\oplus (\zz_{19})^2$.  Taking $\zz_7$ coefficients we have 
$$H_1(B_3;\zz_7)=R_{x-2}\oplus R_{x-4}.$$
The reduced twisted polynomial associated to the character that vanishes on $R_{x-4}$ equals $223t^2-44t+223$ and the reduced twisted polynomial associated to the character that vanishes on $R_{x-2}$ equals  
$1063t^2-3166t+1063$. These are irreducible, and, in particular, not norms.  Hence $P(3,5,-3,-5,7)$ is not slice.

As  a more substantial example, consider the pretzel knot $P(3,7,9,11,15)$.  Permuting the parameter values results in $5! = 120 $ knots, all mutants of each other. However, cyclically permuting the parameter values does not change the isotopy class of the knot, and thus we consider only those permutations that fix the first parameter value at 3.  This reduces us to 24 mutants.   
 The knot $P(3,a,b,c,d)$ can be seen to be the reverse of $P(3,d,c,b,a)$, so this reduces us to 12 mutants and their reverses.  
 In Table~\ref{twistedpolys1}  the twelve we focus on are listed.

  \begin{theorem} The 24   pretzel knot mutants of $P(3,7,9,11,15)$ represent distinct classes in the concordance group.
 
 \end{theorem}

As will be seen, the polynomials become fairly large in studying these knots, so we will only outline the approach and give some specific examples.

 We use the 
  3--fold cover $X_3$ and the corresponding branched cover $B_3$. The untwisted Alexander polynomial of $X_3$  in the cover is 
\noindent  $$ \Delta_{X_3,   \qq\var}(t)  = 3375000000t^4-9893670443t^3 \qquad   $$ 
  $$\qquad \qquad +13204318970 t^2  -9893670443t+3375000000.$$ 
This can be computed from the Alexander polynomial $1500-5807\,t+8615\,{t}^{2}-5807\,{t}^{3}+1500\,{t}^{4} $ of  $P(3,7,9,11,15)$ by using Corollary~\ref{corpoly}. Theorem \ref{bigrep} implies that this polynomial equals $\Delta_{X,\tilde{\rho}_0}(t)$ (and so also $\tD_{X,\tilde{\rho}_0}(t)$), where $\tilde{\rho}_0$ corresponds to the zero character $\chi\co H_1(B_3)\to \zz_q$.

The first homology satisfies  $H_1(B_3) \cong T \oplus T$,
 where $T = \zz_2 \oplus \zz_{7} \oplus \zz_{13} \oplus \zz_{71}$. 
 The 24 mutants are distinguished in the concordance group by the twisted polynomials associated to the choices $q=7$ and $q=13$.

We begin by focusing on the the 7--torsion, 
$H_1(B_3)_{(7)} \cong  H_1(B_p;\zz_7)$, 
so we set $q = 7$.  In this case the homology splits into the direct sum 
$$H_1(B_3;\zz_7)\cong R_{x-2}\oplus R_{x-4}.$$
This is precisely the case that occurred in analyzing a single 4--stranded pretzel knot in~\cite{kl2} and the analysis is much the same.  The main distinction is that because of the added complexity here, computing the twisted Alexander polynomials via a presentation of the fundamental group of the 3--fold cover would be daunting.  The computation is made accessible using Theorem \ref{bigrep}.

Table~\ref{twistedpolys1}  lists the twisted polynomials associated to the nontrivial representations  that factor through either $R_{x-2}$ or $R_{x-4}$.  Fix one of each and denote them $\rho_2$ and $\rho_4$.  Note, reversing the orientation of a knot interchanges $R_{x-2}$ and $R_{x-4}$. Since the  $\tD_{X,\rho_i}$ are all integer (rather than $\qq[\zeta_q]$) polynomials, the Galois automorphism $\zeta_q\mapsto\zeta_q^a$ leaves $\tD_{X,\rho_i}$ fixed if $a\ne 0$. Hence $\tD_{X,\rho_i}$ is independent of the choice of nonzero character $\chi:R_{x-i}\to \zz_7$.

\begin{table}[h]\label{twistedpolys1}  
$$\begin{array}{ccc}
  
\text{Knot} & \tD_{X, \tilde{\rho}_2} & \tD_{X, \tilde{\rho}_4} \\  \hline
P(3,7,9,11,15)  & -8000  t^2+12519  t-8000& 5713  t^2-8194  t+5713     \\    
 P(3,15,7,9,11)  &-438976+826423  t-438976  t^2&      t^2+24  t+1  \\     
P(3,7,15,9, 11)&   -438976+826423  t-438976  t^2 &    t^2+24  t+1     \\     
P(3,7,9,15,11)&    -438976+826423  t-438976  t^2  &     t^2+24  t+1   \\      
 
 P(3,9,11,15,7)&  -125  t^2-88  t-125   & -59443  t^2+102315  t-59443       \\      
  
 P(3,9,11,7,15)&  -125  t^2-88  t-125   & -59443  t^2+102315  t-59443       \\      

 P(3,15,9,11,7)& -314432  t^2+547256  t-314432    &     64  t^2-305  t+64   \\

 P(3,9,15,11,7)&-314432  t^2+547256  t-314432      &  64  t^2-305  t+64     \\      

 P(3,15,11,7,9)&  5713  t^2-8194  t+5713   &   -8000  t^2+12519  t-8000     \\      

 P(3,11,15,7,9)&     t^2+24  t+1 &  -438976+826423  t-438976  t^2      \\      

 P(3,11,7,15,9)&   t^2+24  t+1  &       -438976+826423  t-438976  t^2  \\      

 P(3,11,7,9,15)&     t^2+24  t+1  &  -438976+826423  t-438976  t^2    \\      \hline

\end{array}
$$
\caption{}
\end{table}

Let $P_1$ and $P_2$ be two of the knots listed or their reverses.  If they were concordant, then $P_1 \# -P_2$ would be slice.  Thus, there would be a 2--dimensional invariant metabolizer in the $\zz_7$  homology of the 3--fold cover so that all associated Casson-Gordon invariants would vanish, and in particular the corresponding reduced twisted polynomials would be norms.  Let $X_i$ denote the complement of $P_i$.

Since the homology of a branched cover of a connected sum of knots is naturally the direct sum of the homology of the summands, the homology of the 3--fold branched cover of $P_1 \# -P_2$ with $\zz_7$--coefficients is isomorphic to 
\begin{equation}\label{decompofsum}
R_{x-2}\oplus R_{x-4}\oplus R_{x-2}\oplus R_{x-4}.\end{equation}
Proceeding as in~\cite{kl2} we find that if $P_1\# -P_2$ were slice, certain  products would be be norms in $\qq[\zeta_3]\var$. The possibilities are:

\begin{itemize}
\item If the image $\bar{A}$ of the invariant metabolizer in the homology of the 3--fold branched cover equals   $0\oplus R_{x-4}\oplus 0\oplus R_{x-4}$ in the decomposition (\ref{decompofsum}),  take a character $\chi$ which is nontrivial on the first summand and trivial on the other three, and hence vanishes on $\bar{A}$.  The resulting reduced twisted polynomial equals the product  $ \tD_{X_1, \tilde{\rho}_2}(t)  \tD_{X_2, \tilde{\rho}_0}(t)$.

\item If the image of the invariant metabolizer in the homology of the 3--fold branched cover equals   $R_{x-2}\oplus 0\oplus R_{x-2}\oplus 0$,  take a character $\chi$ which is nontrivial on the second summand and trivial on the three.  The resulting reduced twisted polynomial equals the product  $ \tD_{X_1, \tilde{\rho}_4}(t)  \tD_{X_2, \tilde{\rho}_0}(t)$.

\item Since $\bar{A}$ is invariant, the only other possibility is that $\bar{A}$ is spanned by a pair of vectors of the form $(a,0,b,0)$ and $(0,c,0,d)$.   If $a$ and $b$ are both nonzero, define $\chi$ to be the dot product with $(-b,0,a,0)$, This vanishes on $A$ and  the resulting  reduced twisted polynomial equals the product 
$ \tD_{X_1, \tilde{\rho}_2}(t)  \tD_{X_2, \tilde{\rho}_2}(t)$.  Similarly if both $c$ and $d$ are nonzero 
one finds a character vanishing on $\bar{A}$ with reduced twisted polynomial  $ \tD_{X_1, \tilde{\rho}_4}(t)  \tD_{X_2, \tilde{\rho}_4}(t)$. If one of $a,b,c$ or $d$ is zero one can choose $\chi$ as in the first two cases.

\end{itemize}

The first two cases do not produce norms for any of the knots.  The third case clearly does, as one can see from the table. For example, the calculations do  not rule out the possibility that $P(3,15,7,9,11)$ 
is concordant to the reverse of $P(3,11,15,7,9)$.  The calculations with $q=7$ therefore do not rule out 
the possibility that some pairs of the 12 knots or their reverses   might be concordant.  

To eliminate the possibility of concordance of these remaining pairs, we   calculate with $q=13$. It turns out that the pairs not distinguished by the $q=7$ twisted polynomials are distinguished by the $q=13$ polynomials.

  In this case
$$H_1(B_3;\zz_{13})\cong R_{x-3}\oplus R_{x-9}.$$  The analysis is similar to that in the previous examples, quickly reducing to the third case.  However, now the polynomials have coefficients that are in $\qq[\zeta_{13}]$, but not in $\zz$ or $\qq$.   One must therefore consider the polynomials as well as those obtained by taking  Galois conjugates of the coefficients, since the twisted polynomial of a multiple $a\chi$ of $\chi$  is obtained from the twisted polynomial for $\chi$ by applying the Galois automorphism $\zeta_{13}\mapsto \zeta_{13}^a$.  These polynomials are quite long   and we only indicate one example.

 Consider $P_1 = P(3,15,7,9,11)$ and $P_2 = P(3,7,15,9,11)$.  The $q=7$ calculations do not rule out the possibility that these are concordant.  The polynomials $ \tD_{X, \tilde{\rho}_i}(t)  $ with $q = 13$ and $i = 3$ or $9$   for both of these knots are  irreducible, quadratic, symmetric and can be made monic.  Then  each is determined by its linear coefficient.  That is,  $  \tD_{X, \tilde{\rho}_i}(t)  = t^2 + c t+1$.  We give  the values of only  $c$, with $\zeta_{13}$ abbreviated $\zeta$.\vskip.1in
 
$ \tD_{X_1, \tilde{\rho}_3}(t), c =
\frac{5023889}{3319616}\zeta^{12}+\frac{4735277}{3319616}\zeta^{11}+\frac{5023889}{3319616}\zeta^{10}+\frac{5023889}{3319616}\zeta^9+\frac{4735277}{3319616}\zeta^8+\frac{4735277}{3319616}\zeta^7+\frac{4735277}{3319616}\zeta^6+\frac{4735277}{3319616}\zeta^5+\frac{5023889}{3319616}\zeta^4+\frac{5023889}{3319616}\zeta^3+\frac{4735277}{3319616}\zeta^2+\frac{5023889}{3319616}\zeta
 $.
\vskip.1in

$  \tD_{X_1, \tilde{\rho}_9}(t)  , c =
\frac{1130}{79}\zeta^{12}+\frac{626}{79}\zeta^{11}+\frac{1130}{79}\zeta^{10}+\frac{1130}{79}\zeta^9+\frac{626}{79}\zeta^8+\frac{626}{79}\zeta^7+\frac{626}{79}\zeta^6+\frac{626}{79}\zeta^5+\frac{1130}{79}\zeta^4+\frac{1130}{79}\zeta^3+\frac{626}{79}\zeta^2+\frac{1130}{79}\zeta$.

\vskip.1in

$ \tD_{X_2, \tilde{\rho}_3}(t) , c =-\frac{511538}{55171}\zeta^{12}-\frac{271466}{55171}\zeta^{11}-\frac{511538}{55171}\zeta^{10}-\frac{511538}{55171}\zeta^9-\frac{271466}{55171}\zeta^8-\frac{271466}{55171}\zeta^7-\frac{271466}{55171}\zeta^6-\frac{271466}{55171}\zeta^5-\frac{511538}{55171}\zeta^4-\frac{511538}{55171}\zeta^3-\frac{271466}{55171}\zeta^2-\frac{511538}{55171}\zeta$.

\vskip.1in

$ \tD_{X_2, \tilde{\rho}_9}(t)  , c =-\frac{97030}{1327}\zeta^{12}-\frac{172810}{1327}\zeta^{11}-\frac{97030}{1327}\zeta^{10}-\frac{97030}{1327}\zeta^9-\frac{172810}{1327}\zeta^8-\frac{172810}{1327}\zeta^7-\frac{172810}{1327}\zeta^6-\frac{172810}{1327}\zeta^5-\frac{97030}{1327}\zeta^4-\frac{97030}{1327}\zeta^3-\frac{172810}{1327}\zeta^2-\frac{97030}{1327}\zeta$.

\vskip.1in
Notice that we have written these numbers in terms of the powers $\zeta^1, \cdots , \zeta^{12}$.  Thus, the action of the Galois group permutes the coefficients, and it is easy to see that none of these are conjugate to each other. Arguing as in the third case above one concludes that $P_1$ and $P_2$ are not concordant.

The complete analysis proceeds in the same manner.

\newcommand{\etalchar}[1]{$^{#1}$}
 
\end{document}